\documentclass[11pt]{amsart}
\usepackage{mathrsfs,amssymb}
\usepackage{graphicx}
\begin{document}

\newtheorem{theorem}{Theorem}[section]
\newtheorem{proposition}[theorem]{Proposition}
\newtheorem{lemma}[theorem]{Lemma}
\newtheorem{corollary}[theorem]{Corollary}
\newtheorem{conjecture}[theorem]{Conjecture}
\newtheorem{question}[theorem]{Question}
\newtheorem{problem}[theorem]{Problem}
\theoremstyle{definition}
\newtheorem{definition}{Definition}

\theoremstyle{remark}
\newtheorem{remark}[theorem]{Remark}

\renewcommand{\labelenumi}{(\roman{enumi})}
\def\theenumi{\roman{enumi}}

\numberwithin{equation}{section}

\renewcommand{\Re}{\operatorname{Re}}
\renewcommand{\Im}{\operatorname{Im}}

\def\scrA{{\mathcal A}}
\def\scrB{{\mathcal B}}
\def\scrD{{\mathcal D}}
\def\scrL{{\mathcal L}}
\def\scrS{{\mathcal S}}

\def \G {{\Gamma}}
\def \g {{\gamma}}
\def \R {{\mathbb R}}
\def \H {{\mathbb H}}
\def \C {{\mathbb C}}
\def \Z {{\mathbb Z}}
\def \Q {{\mathbb Q}}
\def \TT {{\mathbb T}}
\newcommand{\T}{\mathbb T}
\def \GinfmodG {{\Gamma_{\!\!\infty}\!\!\setminus\!\Gamma}}
\def \GmodH {{\Gamma\setminus\H}}
\def \vol {\hbox{vol}}
\def \sl  {\hbox{SL}_2(\mathbb Z)}
\def \slr  {\hbox{SL}_2(\mathbb R)}
\def \psl  {\hbox{PSL}_2(\mathbb R)}

\newcommand{\mattwo}[4]
{\left(\begin{array}{cc}
                        #1  & #2   \\
                        #3 &  #4
                          \end{array}\right) }

\newcommand{\rum}[1] {\textup{L}^2\left( #1\right)}
\newcommand{\norm}[1]{\left\lVert #1 \right\rVert}
\newcommand{\abs}[1]{\left\lvert #1 \right\rvert}
\newcommand{\inprod}[2]{\left \langle #1,#2 \right\rangle}
\newcommand{\tr}[1] {\hbox{tr}\left( #1\right)}

\renewcommand{\^}[1]{\widehat{#1}}

\renewcommand{\i}{{\mathrm{i}}}

\date{\today}

\newcommand{\area}{\operatorname{area}}
\newcommand{\ecc}{\operatorname{ecc}}

\newcommand{\Op}{\operatorname{Op}}
\newcommand{\dom}{\operatorname{Dom}}
\newcommand{\Dom}{\operatorname{Dom}}

\newcommand{\Norm}{\mathcal N}

\title[Wave functions for a point scatterer on the torus]{Statistics of wave functions for a point scatterer  on the  torus}
\date{\today}
\author{Ze\'ev Rudnick \and Henrik Uebersch\"ar}

\address{Raymond and Beverly Sackler School of Mathematical Sciences,
Tel Aviv University, Tel Aviv 69978, Israel}
\email{rudnick@post.tau.ac.il}
\address{Raymond and Beverly Sackler School of Mathematical Sciences,
Tel Aviv University, Tel Aviv 69978, Israel}
\email{henrik@post.tau.ac.il}
\date{\today}

\begin{abstract}
Quantum systems whose classical counterpart have ergodic dynamics
are quantum ergodic in the sense that almost all eigenstates are
uniformly distributed in phase space. In contrast, when the
classical dynamics is integrable, there is concentration of
eigenfunctions on invariant structures in phase space. In this paper
we study eigenfunction statistics for the Laplacian perturbed by a
delta-potential (also  known as a point scatterer) on a  flat torus,
a popular model used to study the transition between integrability
and chaos in quantum mechanics. The eigenfunctions of this operator
consist of eigenfunctions of the Laplacian which vanish at the
scatterer, and new, or perturbed, eigenfunctions. We show that
almost all of the perturbed eigenfunctions are uniformly distributed
in configuration space.
\end{abstract}

\maketitle

\section{Introduction}

Quantum systems whose classical counterpart have ergodic dynamics
satisfy Schnirelman's theorem, which asserts that almost all
eigenstates are uniformly distributed in phase space in an
appropriate sense \cite{Schnirelman, CdV2, Zelditch}. In contrast,
when the classical dynamics is integrable, there is concentration of
eigenfunctions on invariant structures in phase space. In this paper
we study eigenfunction statistics for an intermediate system, that
of a point scatterer on the flat torus.

The use of point scatterers, or $\delta$-potentials, goes back to
the Kronig-Penney model \cite{Kronig-Penney} which is  an idealized
solvable model used to explain conductivity in a solid crystal and
the appearance of electronic band structure. They have also been
studied in the mathematical literature to explain the spurious
occurrence of the Riemann zeros in a numerical experiment
\cite{Hejhal}. Billiards with a point scatterer have been used
extensively in the quantum chaos literature, starting with Seba
\cite{Seba}, to model quantum systems strongly perturbed in a region
smaller than the wavelength of the particle.


The   flat torus 
is a standard example of a
system for which the geodesic flow is completely integrable. 
Placing a scatterer at a point $x_0$ in the torus does not change
the classical dynamics except for a measure zero set of
trajectories, and gives a quantum system whose dynamics is generated
by an operator formally written as
\begin{equation}\label{formalop intro}
-\Delta + \alpha \delta_{x_0}
\end{equation}
with $\delta_{x_0}$ being the Dirac mass at $x_0$ and  $\alpha$
being a coupling parameter. Mathematically this corresponds to
picking a self-adjoint extension of the Laplacian $-\Delta$ acting
on functions vanishing near $x_0$ (see Section \S~\ref{sec:pt
scatterer} and Appendix \S~\ref{sec:rigorous}). Such extensions are
parameterized by a phase $\phi\in (-\pi, \pi]$, with $\phi=\pi$
corresponding to the standard Laplacian ($\alpha=0$ in
\eqref{formalop intro}). We denote the corresponding operator by
$-\Delta_{x_0,\phi}$, whose domain consists of a suitable space of
functions $f(x)$ whose behavior near $x_0$ is given by
\begin{equation}
f(x) = C\left( \cos \frac \phi 2 \cdot \frac{\log|x-x_0|}{2\pi}  +
\sin \frac \phi 2\right) +o(1), \quad x\to x_0
\end{equation}
for some constant $C$. For $\phi=\pi$ the eigenvalues are those of
the standard Laplacian.
For $\phi\neq \pi$ ($\alpha\neq 0$) the resulting spectral problem
still has the eigenvalues from the unperturbed problem, with
multiplicity decreased by one,
as well as a new set $\Lambda_\phi$ of eigenvalues interlaced
between the sequence of unperturbed eigenvalues, each appearing with
multiplicity one, and satisfying the spectral equation
\begin{equation}\label{spectral eq}
\sum_n |\psi_n(x_0)|^2(\frac 1{\lambda_n-\lambda} - \frac
{\lambda_n}{\lambda_n^2+1}) = c_0 \tan \frac \phi 2
\end{equation}
for a certain $c_0>0$, where $\{\psi_n(x)\}$ form an orthonormal
basis of eigenfunctions for the unperturbed problem: $-\Delta
\psi_n=\lambda_n\psi_n$.   The  eigenfunction corresponding to
$\lambda\in \Lambda_\phi$ is the Green's function
\begin{equation}
G_\lambda(x;x_0)=(\Delta+\lambda)^{-1}\delta_{x_0} \;.
\end{equation}


%

Our main result is that for almost all $\lambda\in \Lambda_\phi$,
the perturbed eigenfunctions $G_\lambda(\bullet;x_0)$ are uniformly
distributed in position space. To formulate the result precisely, we
denote by
\begin{equation}
g_\lambda(x):=\frac{G_\lambda(x;x_0)}{||G_\lambda||_2}
\end{equation}
 the $L^2$-normalized Green's function:
\begin{theorem}\label{Thm:intro}
Fix $\phi\in (-\pi,\pi)$. There is a subset
$\Lambda_{\phi,\infty}\subset \Lambda_\phi$ of density one so that
for all observables $a\in C^\infty(\TT^2)$,
\begin{equation}
\int_{\TT^2} a(x) g_\lambda(x)^2 dx \to \frac 1{\area(\TT^2)}
\int_{\TT^2} a(x) dx
\end{equation}
as $\lambda \to \infty$ along the subsequence
$\Lambda_{\phi,\infty}$
\end{theorem}

\noindent{\bf Remarks:}


 For the eigenfunctions of the unperturbed Laplacian,
there is a variety of possible limits in the position
representation, which were investigated by Jakobson \cite{Jakobson}.

A result of the same nature as our Theorem~\ref{Thm:intro} was
recently obtained in \cite{MR} for billiards in rational polygons.
There it is shown that for any orthonormal basis of eigenfunctions,
there is a density one subsequence which equidistributes in
configuration space. The method of \cite{MR} adapts the proof of
quantum ergodicity for billiards of \cite{ZZ} to work in
configuration space and inputs the theorem of Kerckhoff, Masur and
Smillie \cite{KMS} who showed that for rational polygons, the
billiard flow is uniquely ergodic in almost every direction. Our
argument here is completely different and is very specific to this
particular model.

A related, and in some sense complementary, issue was studied by
Berkolaiko, Keating and Winn \cite{BKW} who predict that for an
irrational torus with a point scatterer there is a subsequence of
eigenfuctions which "scar" in momentum space, and this was proved by
Keating, Marklof and Winn \cite{KMW} to be the case assuming that
the eigenvalues of the Laplacian on the unperturbed irrational torus
have Poisson spacing distribution, as is predicted by the
Berry-Tabor conjecture.

It is important to note that we (as well as \cite{BKW, KMW}) deal
with the limit of large energy $\lambda \to \infty$ for a fixed
phase $\phi\neq \pi$, which is called the weak coupling limit in the
physics literature. An interesting problem would be to understand
the strong coupling limit, where $\lambda\to \infty$ together with
$\phi\to\pi$ while $\tan(\phi/2)\approx \log \lambda$, so that the
RHS of the spectral equation \eqref{spectral eq} blows up. In that
range it has been argued that the spectrum displays intermediate
statistics \cite{Shigehara, CS, BGeS, BGiS, Rahav-Fishman}.


\noindent{\bf Acknowledgments:} We thank Maja Rudolph for her help
with the numerical investigation of some of these issues, and John
Friedlander for discussions concerning sums of two squares. Z.R. was
partially supported by the Israel Science Foundation (grant No.
1083/10). H.U. was supported by a Minerva Fellowship.


\section{The flat torus} \label{sec:torus}
\newcommand{\lat}{\mathcal L}
\newcommand{\ave}[1]{\left\langle#1\right\rangle} 

\subsection{Basic setup}
We consider a flat torus $\TT^2$ obtained by identifying opposite
sides of a rectangle with side lengths $2\pi /a$, $2\pi a$, so that
$\TT^2 = \R^2/2\pi \lat_0$ where $\lat_0 = \Z(1/a,0)\oplus \Z(0,a)$
is a unimodular lattice.

An orthonormal basis of eigenfunctions for the Laplacian $\Delta$ on
$\TT^2$ consists of the exponentials
\begin{equation}
  \frac 1{2\pi} e^{i\langle x,\xi \rangle}
\end{equation}
where $\xi$ ranges over the dual, or reciprocal, lattice
\begin{equation}
  \lat = \{x\in \R^2: \langle x,\ell \rangle \in \Z, \quad \forall
  \ell \in \lat_0 \}=\{(ma,\frac na):m,n\in \Z\}
\end{equation}

The eigenvalues of the Laplacian on $\TT^2$ are  the {\em norms}
$|\xi|^2$ of the vectors of the dual lattice $\lat$. Weyl's law for
the torus, establishing the asymptotics of the counting function
$N(x)$ of eigenvalues below $x$, is equivalent to  counting the
number of points of the lattice $\lat$ in a disk (equivalently the
number of points of the standard lattice $\Z^2$ in an ellipse), and
therefore reads
\begin{equation}\label{weyl law for torus}
  N(x) = \#\{|\xi|^2\leq x: \xi\in \lat \} = \pi x +O(x^\theta)
\end{equation}
The exponent $\theta$ in the remainder term is known to be at least
$\theta>1/4$. The ``trivial'' bound on the remainder term, as the
length of the boundary, translates into $\theta\leq 1/2$. A
nontrivial bound uses Poisson summation and the method of stationary
phase leads to $\theta\leq 1/3$. We will need a better bound
\begin{equation}
  \theta<\frac 13
\end{equation}
such as the one due to van der Corput  \cite{van der Corput}.
The current world record of $\theta \leq 131/416+o(1)$ is due to
Huxley \cite{Huxley2003}.

 Using the remainder term \eqref{weyl law for torus}, we may deduce
 a bound for the number of lattice points in an annulus: Define
 \begin{equation}
 A(\lambda,L) = \{\xi\in \lat: \lambda-L<|\xi|^2<\lambda+L\}
 \end{equation}
 Then \eqref{weyl law for torus} implies
 \begin{equation}\label{counting annulus}
\#A(\lambda,L) = 2\pi L +O(\lambda^\theta)
 \end{equation}

\subsection{Multiplicities}

Denote by $\Norm=\{0<n_1<\dots \}$ the set of norms of the dual
lattice vectors. The multiplicities in the spectrum are
\begin{equation}
  r_\lat(n) = \#\{\xi\in \lat: |\xi|^2 = n \}
\end{equation}

The lattice $\lat$ is {\em rational} if, after a suitable scaling,
the  norms $|\xi|^2$ are all  rational. The norm of an arbitrary
lattice vector $\xi = (ma,n/a)$ is
\begin{equation}
  |\xi|^2 = a^2m^2+n^2/a^2 = \frac 1{a^2}(a^4m^2+n^2)
\end{equation}
so that the lattice is rational if and only if $a^4\in \Q$ is
rational.

In the irrational case, the multiplicities are entirely due to the
reflection symmetries $(x,y)\mapsto (\pm x, \pm y)$ and are given by
\begin{equation}\label{multiplicities}
  r_\lat(n) = \begin{cases} 1,&n=0\\2,&n=a^2u^2 \mbox{ or } v^2/a^2, u,v\in \Z\\4,&\mbox{ otherwise} \end{cases}
\end{equation}
and in particular are generically equal to $4$.

For the rational case the multiplicities are complicated arithmetic
functions. For instance in the case of the standard lattice
$\lat=\Z^2$, the multiplicity $r(n) = \#\{(x,y)\in \Z^2: x^2+y^2=n
\}$ depends on the prime factorization of the integer $n$. In any
case, it is well known that we have an upper bound on the
multiplicities of the form (see e.g. the proof of \cite[Lemma
7.2]{ORW})
\begin{equation}\label{bd for mult}
  r_\lat(n) \ll n^{o(1)}
\end{equation}

The counting function of the norms is
\begin{equation}
  \Norm(x):=\#\{n\in \Norm: n\leq x\}
\end{equation}
Since we have the upper bound \eqref{bd for mult} and since we know
that the sum over $n\leq x$ of the multiplicities is asymptotically
$\pi x$ \eqref{weyl law for torus}, we deduce a lower bound
\begin{equation}\label{lower bd on N(x)}
  \Norm(x) \gg x^{1-o(1)}
\end{equation}

\subsection{Nearest neighbor gaps}
If the norms are ordered by $\Norm = \{0<n_1<n_2<\dots\}$, we need
to understand the spacings (or gaps) $n_{k+1}-n_k$ between
successive norms. The individual values are difficult to understand.
From \eqref{weyl law for torus} we certainly have $  n_{k+1}-n_k \ll
n_k^\theta$. However one can do better by arguing as follows
\cite{BC}: First find the largest integer square $ u^2<n_{k+1}/a^2$,
which  one can do so that $n_{k+1}-a^2u^2\ll \sqrt{n_{k+1}}$. After
that find the largest square $v^2/a^2< n_{k+1}-a^2u^2$, which one
can do so that $n_{k+1}-a^2u^2-v^2/a^2\ll \sqrt{ n_{k+1}-a^2u^2}\ll
n_{k+1}^{1/4}$. Thus we found a norm $n=a^2u^2+v^2/a^2  \in \Norm$
with $n<n_{k+1}$ so that $n\leq n_k$ giving
\begin{equation}\label{upper bd for gaps}
   n_{k+1}-n_k \ll n_k^{1/4}
\end{equation}

The average spacing for norms up to $x$ is, using \eqref{upper bd
for gaps},
\begin{equation}
  \frac 1{\Norm(x)} \sum_{n_k\leq x} ( n_{k}-n_{k-1} ) =
\frac{x(1+o(1))}{\Norm(x)}
\end{equation}
and by the lower bound \eqref{lower bd on N(x)} we deduce that
\begin{equation}
\frac 1{\Norm(x)} \sum_{n_k\leq x} ( n_{k}-n_{k-1} )\ll x^{o(1)}
\end{equation}
Since we are dealing with averages of positive quantities, we find:
\begin{lemma}\label{lem:gap}
  For almost all $k$, that is on a density one sequence, the spacings satisfy
  \begin{equation}
    n_{k+1}-n_k \ll n_k^{o(1)}
  \end{equation}
\end{lemma}


\section{point scatterers and $\delta$-potentials}
\label{sec:pt scatterer}

In this section we review the theory of point scatterers  (see
\cite{CdV}), with some details left to Appendix~\ref{sec:rigorous}.

\subsection{A finite-dimensional model}

We want to study the Schr\"odinger operator with a delta-potential
on the flat $D$-dimensional torus  $\TT^D=\R^D/2\pi \lat_0$, where
$\lat_0\subset \R^D$ is a unimodular lattice. The operator is
formally given by
\begin{equation}\label{formal operator}
-\Delta + \alpha \delta
\end{equation}
where $\delta$ is the Dirac delta-function at the point $x_0$.

To make sense of the operator \eqref{formal operator}, we say that
for a wave function $\psi\in C^\infty(\TT^D)$, multiplication by
$\delta$ should give
\begin{equation}
 \delta  \psi = \psi(x_0)\delta = \langle \psi,\delta \rangle \delta
\end{equation}
which we try to think of as a rank-one operator.

As an approximation, it is useful to examine a finite-dimensional
model: a rank one perturbation of a self-adjoint operator $H_0$ on a
finite-dimensional Hilbert space $\mathcal H$ of the form
\begin{equation}
  H = H_0 + \alpha P_v, \quad P_vu =\langle u,v \rangle v,\quad u\in \mathcal H
\end{equation}
where $0\neq v\in \mathcal H$ and $\alpha\neq 0$ is real. Let
$\phi_n$ be an orthonormal basis of $\mathcal H$ consisting of
eigenvectors  $H_0$ , with eigenvalues $\epsilon_n$: $H_0 \phi_n =
\epsilon_n \phi_n$. The eigenvectors $\phi_n$ of $H_0$ with $\langle
\phi_n,v\rangle =0$ are clearly still eigenvectors of $H$. The new
part of the spectrum is given by:

\begin{lemma}
The eigenvalues $E\notin \mbox{Spec }(H_0)$ of the perturbed
operator $H$ are the solutions of the equation
\begin{equation}\label{model eigenvalue eq}
  \langle (E-H_0)^{-1} v,v \rangle = \frac 1 \alpha
\end{equation}
or equivalently
\begin{equation}\label{divergent sum}
\sum_n \frac{ |\langle v,\phi_n \rangle |^2}{E-\epsilon_n} = \frac
1\alpha
\end{equation}
with corresponding eigenfunction
\begin{equation}\label{new eigenfunction model}
u=  (E-H_0)^{-1}v
\end{equation}
\end{lemma}

\begin{proof}
We rewrite the eigenvalue equation $Hu=Eu$ for $H$  in the form
\begin{equation}\label{eigenvalue equation}
  (E-H_0)u = \alpha \langle u,v \rangle v
\end{equation}
If $E\notin \mbox{Spec}(H_0)$ then necessarily $\langle u,v\rangle
\neq 0$ and we find that
\begin{equation}
  u = \alpha  \langle u,v \rangle (E-H_0)^{-1} v \
\end{equation}
Thus up to a scalar multiple
\begin{equation} \label{condition on u}
u=  (E-H_0)^{-1}v
\end{equation}

Substituting \eqref{condition on u} in the eigenvalue equation
\eqref{eigenvalue equation} gives
\begin{equation}
v=   \alpha \langle (E-H_0)^{-1} v,v \rangle v
\end{equation}
that is
\begin{equation}
  \langle (E-H_0)^{-1} v,v \rangle = \frac 1 \alpha
\end{equation}
Expanding $v=\sum_n \langle v,\phi_n \rangle \phi_n$ in terms of the
normalized eigenvectors $\phi_n$ of $H_0$  gives
\begin{equation}
\sum_n \frac{ |\langle v,\phi_n \rangle |^2}{E-\epsilon_n} = \frac
1\alpha
\end{equation}

Conversely, if $E\notin \mbox{Spec}(H_0)$  and \eqref{model
eigenvalue eq}  holds, take $u=(E-H_0)^{-1}v$ as in \eqref{new
eigenfunction model}. Then
$$
(H-E)u = (H_0-E)u+\alpha \langle u,v \rangle v
$$
Since $(H_0-E)u = (H_0-E)(E-H_0)^{-1}v=-v$ and $\alpha \langle u,v
\rangle v = \alpha \langle (E-H_0)^{-1}v,v\rangle v = v$ by
\eqref{model eigenvalue eq} we find that $Hu=Eu$.
\end{proof}

Now take for $H_0$ the free Schr\"odinger operator $H_0 = -\Delta$
acting  on $ C^\infty(\TT^D)\subset L^2(\TT^D)$,
 and $v=\delta_{x_0}$.
Then the eigenfunctions \eqref{new eigenfunction model} with
eigenvalues $E\notin \mbox{Spec }( H_0)$ are  the Green's function
\begin{equation}
  G_E(\bullet;x_0) = (E-H_0)^{-1} \delta_{x_0}
\end{equation}
However the sum \eqref{divergent sum} diverges; indeed,in that case
the RHS of \eqref{divergent sum}  equals
\begin{equation}\label{divergent sum2}
  \sum_n \frac{ |\phi_n(x_0) |^2}{E-\epsilon_n}
\end{equation}
which is divergent  in dimension $D>1$, by Weyl's law. In fact
\eqref{divergent sum2} is just the Green's function evaluated on the
diagonal,
which is divergent in dimension $D>1$. Thus one needs a
regularization procedure.

\subsection{Regularization}
One regularization procedure is through the theory of self-adjoint
extensions. A succinct account of this procedure is given in
\cite{CdV}. For the reader's convenience this will be reviewed in
Appendix \S~\ref{sec:rigorous}. One starts with the standard
Laplacian $\-\Delta$, for which an orthogonal basis of
eigenfunctions are the exponentials $e^{i\langle \xi,x
  \rangle}$, $\xi$ in the dual lattice $\lat$ with corresponding eigenvalue $|\xi|^2$.
The idea is that for  functions vanishing at the point $x_0$ all
candidates have to coincide with the unperturbed operator $H_0 =
-\Delta$. We want to extend it to a  bigger space. Denoting by
$-\Delta_{x_0}$ the unperturbed operator restricted to
$C_c^\infty(\TT^D\backslash\{x_0\})$, one finds that the adjoint has
as its domain $ \dom(-\Delta_{x_0}^* )$ the Sobolev space
$H^2(\TT^D\backslash\{x_0\})$. In dimension $D=2,3$ this
equals\footnote{In dimension $D=1$ one wants $\Delta
  f-c_0\delta-c_1\delta'\in  L^2(\TT^1)$, while in dimensions $D\geq
  4$, $H_0$ is essentially self adjoint and there are no self-adjoint
  extensions.}
the space of $f\in L^2(\TT^D)$ for which   $\exists A\in \C$  s.t.
\begin{equation}\label{domain adjoint}
\Delta f-A \delta_{x_0} \in L^2(\TT^D)
\end{equation}
and for such $f$, there is some $B\in \C$ so that for $x$ near
$x_0$,
\begin{equation}
f(x) = A G^{(D)}(|x-x_0|)+B +o(1),
\end{equation}
where
\begin{equation}
 G^{(D)}(r) = \begin{cases} \frac 1{2\pi} \log r,& D=2\\
-\frac 1{4\pi  r},& D=3
\end{cases}
 \;.
\end{equation}

One finds that in dimension $D=2,3$ there is a one-parameter family
of extensions, parameterized by $\phi \in \R/2\pi \Z\simeq U(1)$,
denoted by $-\Delta_{\phi,x_0}$ with domain
given by $f\in  \dom(-\Delta_{x_0}^* )$ for which there is some
$a\in\C$ with
\begin{equation}
f(x) = a\left( \cos \frac \phi 2 \cdot  G^{(D)}(|x-x_0|) + \sin
\frac \phi 2\right) +o(1), \quad x\to x_0
\end{equation}
The action of $-\Delta_{\phi,x_0}$ on $f$ satisfying \eqref{domain
adjoint} is then given by
\begin{equation}
 - \Delta_{\phi,x_0}f = -\Delta f +A \delta_{x_0}
\end{equation}
The divergent equation \eqref{divergent sum} is replaced by the
convergent equation for the new eigenvalues $\lambda\notin
\sigma(\Delta)$:
\begin{equation}\label{new eigenvalue equation}
  \sum_{\xi\in \lat} \left\{ \frac 1{|\xi|^2-\lambda} - \frac
      {|\xi|^2}{|\xi|^4+1} \right\}  = c_0 \tan \frac \phi 2
\end{equation}
where
$$c_0=\sum_{\xi\in\lat}\frac{1}{|\xi|^4+1}.$$
We can rewrite \eqref{new eigenvalue equation} as
\begin{equation}\label{new eigenvalue equation bis}
  \sum_{n\in \Norm} r_\lat(n)\left\{ \frac 1{n-\lambda}-\frac{n}{n^2+1}
  \right\} = c_0 \tan \frac \phi 2
\end{equation}
where
$$r_\lat(n)=\#\{n=|\xi|^2:\xi\in \lat\}$$
and $\Norm$ is the set of norms of vectors in $\lat$ (without
multiplicity). The corresponding eigenfunction is a multiple of
$G_\lambda(x;x_0) = (\Delta+\lambda)^{-1}\delta_{x_0}$, which has
the $L^2$-expansion
\begin{equation}
G_{\lambda}(x,x_{0})=-\frac{1}{4\pi^2}\sum_{\xi\in\lat}\frac{\exp(\i\xi\cdot
(x-x_{0}))}{|\xi|^{2}-\lambda},\qquad x\neq x_{0}.
\end{equation}

As may be seen from \eqref{new eigenvalue equation bis}, the new
eigenvalues interlace between the sequence $\Norm=\{n_1<n_2<\dots\}$
of norms, see Figure~\ref{figspectral}. We may thus label the
perturbed eigenvalues $\lambda_k=\lambda_k^\phi$ so that
\begin{equation}\label{order the evs}
  n_1<\lambda_1<n_2<\lambda_2<\dots<n_k<\lambda_k<n_{k+1}<\dots
\end{equation}
\begin{figure}[h]
\begin{center}
  \includegraphics[width=110mm]
{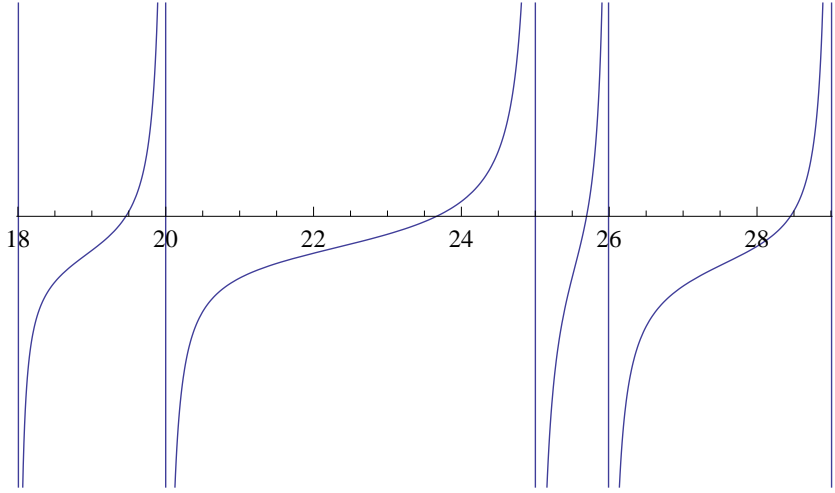}
 \caption{ A plot of the spectral function on the LHS
 of the eigenvalue equation \eqref{new eigenvalue equation bis} for the standard lattice $\lat=\Z^2$.
 The intersections of the plot with the $x$-axis are the perturbed
 eigenvalues corresponding to $\phi=0$,
 alternating with the norms.}
 \label{figspectral}
\end{center}
\end{figure}

\subsection{The density of states}
The density of states of the perturbed eigenvalues depends strongly
on the particular torus, that is on the lattice it determines.

For the standard lattice $\Z^2$, a theorem of Landau \cite{Landau}
asserts that
\begin{equation}
\label{Landau's thm}
  \#\{n\in \Norm, n\leq x\} \sim B\frac{x}{\sqrt{\log x}},
 \end{equation}
where  $B=\frac 1{\sqrt{2}} \prod (1-p^{-2})^{-1/2} = 0.764\dots$,
the product over primes $p=3 \bmod 4$. Consequently we deduce a form
of Weyl's law for the perturbed spectrum $\Lambda_\phi$ of
$\Delta_{\phi,x_0}$ for the standard lattice:
\begin{equation}\label{Weyl's law}
  \#\{\lambda\in \Lambda_\phi: \lambda \leq x \} \sim B\frac{x}{\sqrt{\log x}}
\end{equation}

In the irrational case, the multiplicities are typically $4$, see
\eqref{multiplicities}. Then Weyl's law in those cases would read as
\begin{equation}\label{irrational Weyl's law}
  \#\{\lambda\in \Lambda_\phi: \lambda \leq x \} \sim \frac \pi 4 x
\end{equation}

\section{The norm of $G_\lambda$}
We will need a lower bound on the $L^2$-norm of the Green's function
$G_\lambda$. We are able to get a good bound for a sub-sequence of
density one. To define this subsequence, we recall our discussion of
the gaps between consecutive norms.


According to Lemma~\ref{lem:gap}, for almost all $k$ we have
\begin{equation}\label{small gaps}
   n_{k+1}-n_k \ll n_k^{o(1)}
\end{equation}

We define the set $\Lambda_g\subset \Lambda$ of eigenvalues
$\lambda_k$ (using the labeling \eqref{order the evs}) so that
\eqref{small gaps} holds:
\begin{equation} \label{def Lambda_g}
\Lambda_g=\{\lambda_k\in \Lambda: n_{k+1}-n_k \ll \lambda_k^{o(1)}
\}
\end{equation}
By the discussion above, this is a set of density one in $\Lambda$
(and conjecturally all of $\Lambda$).


\begin{lemma}\label{lem:lower bd on norm}
  For $\lambda\in \Lambda_g$ (i.e. for almost all $\lambda$), we have
  \begin{equation}\label{lowerbound for |G|}
    ||G_\lambda|| \gg \frac 1{\lambda^{o(1)}}
  \end{equation}
\end{lemma}
\begin{proof}
Let $n_k=n_k(\lambda)$, $n_{k+1}=n_{k+1}(\lambda)$ be consecutive
norms so that $n_k<\lambda<n_{k+1}$. Then trivially
\begin{equation}
\left\|G_{\lambda}\right\|^{2}_{2}\gg \sum_{n\in \Norm}
\frac{r_\lat(n)}{(n-\lambda)^2} \geq \frac 1{(n_k-\lambda)^2} >
\frac 1{(n_{k+1}-n_k)^2}
\end{equation}
Since for $\lambda\in \Lambda_g$ we know that $n_{k+1}-n_k\ll
n_k^{o(1)}$,   \eqref{lowerbound for |G|} follows.
\end{proof}

It is natural to conjecture that \eqref{lowerbound for |G|} holds
for all $\lambda$.


\section{Truncation}

For $L>0$ let $A(\lambda,L)$ be the set of lattice points in the
annulus $\lambda-L<|x|^2<\lambda+L$:
\begin{equation}\label{annulus}
  A(\lambda,L) = \{\xi\in \lat: \left| |\xi|^2-\lambda \right|< L\}
\end{equation}
We denote the truncated Green's function by
\begin{equation}
G_{\lambda,L}(x,x_{0})=-\frac{1}{4\pi^2}\sum_{\xi\in
A(\lambda,L)}\frac{\exp(\i\xi\cdot (x-x_{0}))}{|\xi|^{2}-\lambda}.
\end{equation}
Let $g_\lambda$ and $g_{\lambda,L}$ be the $L^2$-normalized Green's
function and its truncation:
\begin{equation}
  g_\lambda = \frac{G_\lambda}{||G_\lambda||} ,\quad
 g_{\lambda,L} = \frac{G_{\lambda,L}}{||G_{\lambda,L}||}
\end{equation}

We have the following approximation.

\begin{lemma}\label{approx}
Let $L=\lambda^{\delta}$, $\theta/2<\delta<1$. As $\lambda\to\infty$
along $\Lambda_g$,
$$\left\|g_{\lambda}-g_{\lambda,L}\right\|_{2}\to 0.$$
\end{lemma}

\begin{proof}
Note that
\begin{equation}
\begin{split}
\left\|\frac{G_\lambda}{\left\|G_\lambda\right\|_2}-\frac{G_{\lambda,L}}{\left\|G_{\lambda,L}\right\|_2}\right\|_2
\leq&\; \frac{\left\|G_\lambda-G_{\lambda,L}\right\|_2}{\left\|G_\lambda\right\|_2}\\
&+\left\|G_{\lambda,L}\right\|_2\left|\frac{1}{\left\|G_\lambda\right\|_2}-\frac{1}{\left\|G_{\lambda,L}\right\|_2}\right|\\
\leq&\;
2\,\frac{\left\|G_\lambda-G_{\lambda,L}\right\|_2}{\left\|G_\lambda\right\|_2}.
\end{split}
\end{equation}


We have
\begin{equation}
  \left\|G_{\lambda}-G_{\lambda,L}\right\|_{2}^{2}=
\frac{1}{16\pi^4} \sum_{||\xi|^2-\lambda|\geq\lambda^\delta}
\frac{1}{(|\xi|^2-\lambda)^2}
\end{equation}
We recall how to evaluate lattice sums using summation by parts:

Let $n_1<n_2<\dots$ be the set of norms, and
\begin{equation}
  N(t) = \sum_{n_k\leq t} r_\lat(n_k)
\end{equation}
Then for a smooth function $f(t)$ on $\R$ we have
\begin{equation}
  \sum_{n_A<|\xi|^2 \leq n_B} f(|\xi|^2) =
N(n_B)f(n_B)-N(n_A)f(n_{A+1})- \int_{n_{A+1}}^{n_B} f'(t)N(t)dt
\end{equation}
Now use the lattice count with remainder \eqref{weyl law for torus}
to get
\begin{multline}\label{sum by parts}
  \sum_{n_A<|\xi|^2 \leq n_B} f(|\xi|^2) =
\pi\int_{n_{A+1}}^{n_B} f(t)dt \\ +O(n_B^\theta f(n_B) +
n_{A+1}^\theta f(n_A)) +O(\int_{n_{A+1}}^{n_B} |f'(t)|t^\theta dt)
\end{multline}


Applying \eqref{sum by parts} with $f(t)=1/(t-\lambda)^2$, once with
$n_A=n_1$ and $n_B\leq \lambda-L<n_{B+1}$ and then with
$n_{A-1}<\lambda+L\leq n_{A}<n_{A+1}$ and $n_B=\infty$ gives
\begin{equation}
\left\|G_{\lambda}-G_{\lambda,L}\right\|_{2}^{2} \ll \frac 1{L} +
\frac{\lambda^\theta}{L^2}
\end{equation}

Since for $\lambda\in \Lambda_g$ we have $||G_\lambda||^2\gg
1/\lambda^{o(1)}$ by Lemma~\ref{lem:lower bd on norm}, we find
$$
\frac{\left\|G_{\lambda}-G^{L}_{\lambda}\right\|_{2}^{2}}
{\left\|G_{\lambda}\right\|_{2}^{2}}\ll \lambda^{o(1)} \left(\frac
1{L} + \frac{\lambda^\theta}{L^2} \right)
$$
which tends to zero if $\delta>\theta/2$.
\end{proof}

Consequently we may study the matrix elements by replacing
$g_\lambda$ by the truncated version $g_{\lambda,L}$:
\begin{lemma}\label{approxtrunc}
Let $f\in C^\infty(\TT^2)$ and $L=\lambda^\delta$,
$\theta/2<\delta<1$. We have $$|\left\langle f\,
g_\lambda,g_\lambda\right\rangle-\left\langle f\,
g_{\lambda,L},g_{\lambda,L}\right\rangle|\to0$$ as
$\lambda\to\infty$ along $\Lambda_g$.
\end{lemma}
\begin{proof}
Let $f\in C^\infty(\TT^2)$. We define the multiplication operator
$M_f:L^2(\TT^2)\to L^2(\TT^2)$ by $$M_f(g)=f\,g.$$ Since $M_f$ is a
continuous operator on $L^2(\TT^2)$, we have that
$$\left\|g_\lambda-g_{\lambda,L}\right\|_2\to0$$ for $\lambda\in
\Lambda_g$ implies
$$\left\|M_f(g_\lambda-g_{\lambda,L})\right\|_2\to0$$ and hence $$|\left\langle M_f\,g_\lambda,g_\lambda-g_{\lambda,L}\right\rangle|\leq\left\|M_f\right\|_\infty\left\|g_\lambda-g_{\lambda,L}\right\|\to0.$$
If we repeat this, where we switch $g_\lambda$ and $g_{\lambda,L}$,
we obtain $$|\left\langle M_f
g_\lambda,g_\lambda\right\rangle-\left\langle M_f
g_{\lambda,L},g_{\lambda,L}\right\rangle|\to0.$$
\end{proof}

\section{Exceptional eigenvalues and a Diophantine inequality}

Fix a nonzero  vector $0\neq \zeta \in \lat$, and $\delta\in(\tfrac
\theta2,\tfrac 12 -\theta )$ (such $\delta$ exists because
$\theta<1/3$). Let $S_\zeta$ be the set of vectors satisfying
\begin{equation}
  S_\zeta = \{\eta\in \lat : |\langle \eta,\zeta   \rangle |\leq |\eta|^{2\delta} \}
\end{equation}
We define a subset $\Lambda_\zeta \subset \Lambda$ of eigenvalues
\begin{equation}\label{def Lambda_zeta}
\Lambda_\zeta = \{\lambda\in \Lambda:  A(\lambda,\lambda^\delta)\cap
S_\zeta = \emptyset \}
\end{equation}
(recall that $A(\lambda,L)$ are the lattice points $\eta\in \lat$ in
the annulus \eqref{annulus}). Our goal in this section is to show
that
\begin{proposition}
$\Lambda_\zeta$ is a subset of density one in $\Lambda$.
\end{proposition}

\begin{proof}
Let
\begin{equation}
  B_\zeta = \Lambda\backslash \Lambda_\zeta=\{\lambda\in \Lambda:  A(\lambda,\lambda^\delta)\cap S_\zeta \neq \emptyset \}
\end{equation}
We will show that $B_\zeta$ has density zero in the set $\Lambda$ of
all perturbed eigenvalues, in fact
\begin{equation}
\#\{\lambda \in B_\zeta: \lambda \leq X \}  \ll
\frac{X^{1-\delta'}}{|\zeta|}
\end{equation}
for $\delta'=1/2-\theta-\delta>0$.

 We first show that
 \begin{lemma}\label{counting S_zeta}
 \begin{equation}
    \# \{\eta \in S_\zeta : |\eta|^2 \leq X\}
\ll \frac{X^{\frac 12+\delta}}{|\zeta|}
  \end{equation}
 \end{lemma}
 \begin{proof}
Introduce cartesian coordinates with one of the axes in the
direction of the vector $\zeta$ by writing every $x\in \R^2$ as
\begin{equation}
  x = u \frac{\zeta}{|\zeta|} + v \frac {\zeta^\perp}{|\zeta^\perp|}
\end{equation}
where if $\zeta = (ap,\frac qa)$ then $\zeta^\perp = (-\frac qa,
pa)$ is a vector orthogonal to $\zeta$. In these coordinates,
\begin{equation}
  \langle x,\zeta \rangle = u|\zeta|\;, \quad |x|^2 = u^2+v^2
\end{equation}
and our set of lattice points is thus contained in the rectangle
\begin{equation}
  \mathcal R = \{u \frac{\zeta}{|\zeta|} + v \frac
         {\zeta^\perp}{|\zeta^\perp|}: |u|\leq  \frac{X^\delta}{|\zeta|},\quad
         |v|\leq X^{1/2} \}
\end{equation}
Now estimating the number of lattice points in a rectangle is a
simple matter: Putting a fundamental domain $\mathcal F =
\{(ax,y/a):0\leq x,y\leq 1\}$ for the lattice (which has unit area
for the case at hand) centered around each lattice point in
$\mathcal R$, we get a figure whose area is the number of lattice
points in question, and which is contained in a slightly bigger
rectangle whose dimensions are expanded by the diameter $d_\lat =
\sqrt{a^2+1/a^2}$ of $\mathcal F$:
\begin{equation}
  \mathcal R^+ = \{u \frac{\zeta}{|\zeta|} + v \frac
         {\zeta^\perp}{|\zeta^\perp|}: |u|\leq
         \frac{X^\delta}{|\zeta|}+d_\lat ,\quad
         |v|\leq X^{1/2} +d_\lat \}
\end{equation}
Thus we see that $\#\lat \cap \mathcal R$ is bounded by the area of
$\mathcal R^+$, which is:
\begin{equation}
   \area \mathcal R^+ =
  2(\frac{X^{\delta}}{|\zeta|}+d_\lat)\times 2(X^{1/2}+d_\lat)
  =\frac {4X^{1/2 +\delta}}{|\zeta|}+O_\lat(X^{1/2})
\end{equation}
Therefore
\begin{equation}
  \#\{\eta\in S_\zeta:|\eta|^2\leq X \} \leq \area \mathcal R^+ =
\frac{ 4X^{1/2+\delta}}{|\zeta|}+O_\lat(X^{1/2})
\end{equation}
as claimed.
\end{proof}

 Next we define $\Norm_\zeta\subset \Norm$ to be the set of norms $|\eta|^2$ of $\eta\in
S_\zeta$, without multiplicities. We clearly have
\begin{equation}
  \#\{n\in \Norm_\zeta:n\leq X \}  \leq \#\{\eta\in  S_\zeta:|\eta|^2\leq X\}  \ll \frac{X^{1/2+\delta}}{|\zeta|}
\end{equation}
by  Lemma~\ref{counting S_zeta}.

We have a map
\begin{equation}
  \iota: B_\zeta \to \Norm_\zeta
\end{equation}
defined by $\iota(\lambda)$ being the closest $n\in \Norm_\zeta$ to
$\lambda$; if there are two such elements, i.e. $n_-<\lambda<n_+$
with $n_\pm \in \Norm_\zeta$ and $n_+-\lambda=\lambda-n_-$, then set
$\iota(\lambda)=n_+$. Thus we get a well defined map, whose fibers
satisfy
\begin{equation*}
  \iota^{-1}(n) \subseteq \{\lambda\in \Lambda: \exists \eta\in  S_\zeta \cap A(\lambda,\lambda^\delta), |\eta|^2=n \}
  \subseteq \Lambda \cap [n-2n^\delta,n+2n^\delta]
\end{equation*}
 for $n\gg 1$.

Since $\Lambda$ is interlaced between the norms $\Norm$, we have
\begin{equation}
  \#\Lambda \cap [n-2n^\delta,n+2n^\delta] \ll
  \sum_{n-3n^\delta<m<n+3n^\delta} r_\lat(m) =\#A(n,3n^\delta)
\end{equation}
which is the number of lattice points in an annulus. By
\eqref{counting annulus},
\begin{equation}
\#A(n,3n^\delta)\ll  n^\delta+n^\theta
\end{equation}
and hence (since $\delta<\theta$)
\begin{equation}
\#\iota^{-1}(n)\ll n^\theta
\end{equation}
We thus find
\begin{equation}
  \begin{split}
 \#\{\lambda\in B_\zeta:\lambda \leq X\}  &=
\sum_{\substack{n\in \Norm_\zeta\\n\leq X}} \#\iota^{-1}(n) \\
&\ll  X^\theta \#\{n\in \Norm_\zeta:n\leq X\} \ll \frac{
X^{1/2+\delta+\theta} }{|\zeta|}
  \end{split}
\end{equation}
That is
\begin{equation}
   \#\{\lambda\in B_\zeta:\lambda \leq X\}\ll
 \frac{ X^{1-\delta'} }{|\zeta|}
\end{equation}
with $\delta' = 1/2-\theta-\delta>0$.
\end{proof}

\section{Proof of Theorem~\ref{Thm:intro}}

\subsection{Fixed observables}
Fix a nonzero vector $\zeta \in \lat$ and recall the definition
\eqref{def Lambda_zeta},  \eqref{def Lambda_g} of the sets of
eigenvalues $\Lambda_\zeta $ and  $\Lambda_g$; both are of density
one in $\Lambda$ and thus
\begin{equation}
  \Lambda_{g,\zeta} := \Lambda_g \cap \Lambda_\zeta
\end{equation}
 is still a set of density one in $\Lambda$.

We will show that $\langle e_\zeta g_\lambda, g_\lambda \rangle \to
0$ as $\lambda \to \infty$ along $\Lambda_{g,\zeta}$. By
Lemma~\ref{approx} it suffices to show:
\begin{proposition}
Take $\lambda\in\Lambda_{g,\delta}$ and $L=\lambda^\delta$,
$\delta\in(\tfrac \theta 2,\tfrac{1}{2}-\theta)$. Fix nonzero
$\zeta\in\lat$. As $\lambda\to 0$ while   $\lambda
\in\Lambda_{g,\zeta}$,
\begin{equation}
\left\langle e_\zeta\, g_{\lambda,L},g_{\lambda,L}\right\rangle \to
0
\end{equation}
\end{proposition}
\begin{proof}
For $\xi\in\lat $ and $\lambda\in\Lambda_{g,\zeta}$ define
$$c(\xi)=\frac{1}{|\xi|^2-\lambda}.$$
The $L^2$-norm of the truncated Green's function $G_{\lambda,L}$
($L=\lambda^\delta$) is given by
\begin{equation}
\left\|G_{\lambda,L}\right\|_2^2=\frac{1}{16\pi^4}\sum_{\xi\in
A(\lambda,L)}c(\xi)^2
\end{equation}
and hence
\begin{equation}
\left\langle e_\zeta\,
G_{\lambda,L},G_{\lambda,L}\right\rangle=\frac{1}{16\pi^4}\sum_{\xi\in
  A(\lambda,L)}c(\xi)c(\xi-\zeta)\;.
\end{equation}
Cauchy-Schwarz gives
$$|\left\langle e_\zeta\,
G_{\lambda,L},G_{\lambda,L}\right\rangle|^2\leq\left\|G_{\lambda,L}\right\|_2^2\,
\sum_{\xi\in  A(\lambda,L)}c(\xi-\zeta)^2.$$ Note that for
$\lambda\in \Lambda_{\zeta,g}$,
$$|c(\xi-\zeta)|\ll \frac{1}{L}.$$
Indeed
\begin{equation}
  |\xi-\zeta|^2-\lambda   =
|\xi|^2-\lambda -2\langle \xi,\zeta
   \rangle +|\zeta|^2
\end{equation}
and since for $\lambda \in \Lambda_\zeta$ and $\xi \in A(\lambda,L)$
we have $|\langle \xi,\zeta
   \rangle|>|\xi|^{2\delta} \sim L$ we find that
   \begin{equation}
\left| |\xi-\zeta|^2-\lambda \right| \geq  2L(1+o(1))-L-|\zeta|^2
\gg L
   \end{equation}


Using \eqref{counting annulus}  gives the bound
\begin{equation}
\sum_{\xi\in A(\lambda,L)}c(\xi-\zeta)^2\ll \frac{\#
A(\lambda,L)}{L^2}\ll \frac {\lambda^{\theta}}{L^2}
\end{equation}
(recall $\theta\geq 1/4$) so that we find
\begin{equation}
  \left\langle e_\zeta\,G_{\lambda,L},G_{\lambda,L}\right\rangle \ll
||G_\lambda|| \frac{\lambda^{\theta/2}}{L}
\end{equation}

The lower bound  $||G_\lambda||\gg 1/\lambda^{o(1)}$ of
Lemma~\ref{lem:lower bd on norm}  implies
$$
\left\langle e_\zeta\, g_{\lambda,L},g_{\lambda,L}\right\rangle
=\frac{\left\langle e_\zeta\,
  G_{\lambda,L},G_{\lambda,L}\right\rangle}{\left\|G_{\lambda,L}\right\|_2^2}
\ll
\frac{\lambda^{\theta/2+o(1)}}{L}=\frac{\lambda^{\theta/2+o(1)}}{\lambda^\delta}
$$
for $\lambda\in\Lambda_{g,\zeta}$, which tends to zero since
$\delta>\theta/2$.
\end{proof}

\subsection{A diagonalization argument} We have shown that for each
$0\neq \zeta\in \lat$, there is a density one subset
$\Lambda_{g,\zeta}$ of eigenvalues so that $\langle e_\zeta
g_\lambda,g_\lambda \rangle \to 0$ as $\lambda\to \infty$ along
$\Lambda_{g,\zeta}$. It remains to see that there is a density one
subset $\Lambda_\infty \subset \Lambda$ so that for every observable
$a\in C^\infty(\TT^2)$, we have
\begin{equation}\label{convergence for all obs}
  \langle a g_\lambda,g_\lambda \rangle \to \frac 1{\mbox{area}(\TT^2)}
\int_{\TT^2} a(x)dx
\end{equation}
as $\lambda\to \infty$ along $\Lambda_\infty$. We recall the
argument, which can be found e.g. in \cite{CdV2}. For $J\geq 1$, let
$\Lambda_J\subset \Lambda$ be of density one so that for all
$|\zeta|\leq J$, $\langle e_\zeta g_\lambda,g_\lambda \rangle \to 0$
as $\lambda\to \infty$ along $\Lambda_{g,\zeta}$, and in particular
for every trigonometric polynomial $P_J(x) = \sum_{|\zeta|\leq J}
p_\zeta e_\zeta(x)$ we have
\begin{equation}\label{convergence for pols}
   \langle P_J g_\lambda,g_\lambda \rangle \to \frac
   1{\mbox{area}(\TT^2)}\int_{\TT^2} P_J(x)dx
\end{equation}
We may assume that $\Lambda_{J+1}\subseteq \Lambda_J$ for each $J$.
Now choose $M_J$ so that for all $X>M_J$,
\begin{equation}
  \frac 1{\#\{\lambda \in \Lambda: \lambda \leq X \}} \#\{\lambda\in
  \Lambda_J: \lambda \leq X \} \geq 1-\frac 1{2^J}
\end{equation}
and let $\Lambda_\infty$  be such that $\Lambda_\infty \cap
[M_J,M_{J+1}] = \Lambda_J \cap [M_J,M_{J+1}]$ for all $J$. Then
$\Lambda_\infty\cap [0,M_{J+1}]$ contains $\Lambda_J \cap
[0,M_{J+1}]$ and therefore $\Lambda_\infty$ has density one in
$\Lambda$ and \eqref{convergence for pols} holds for $\lambda\in
\Lambda_\infty$. Since the trigonometric polynomials are dense in
$C^\infty(\TT^2)$ in the uniform norm and the probability measures
$|g_\lambda(x)|^2dx$ are continuous with respect to this norm, we
find that \eqref{convergence for all obs} holds.

\appendix

\section{A rigorous description of the point scatterer}
\label{sec:rigorous}

Denote by $\delta_{x_0}$ the Dirac distribution at $x_0\in\T^2$. We
are interested in solutions to the equation
\begin{equation}\label{specproblem}
(\Delta+\lambda)f=\delta_{x_0},\qquad f\in
C^\infty(\T^2\setminus\lbrace x_0\rbrace),\; \left\|f\right\|_2=1
\end{equation}
and its association with the eigenfunctions and eigenvalues of a
family of self-adjoint operators. Consider the domain of
$C^\infty$-functions which vanish in a neighborhood of $x_0$
$$D_0=C^\infty_0(\T^2\setminus\lbrace x_0\rbrace)$$ and denote
$-\Delta_{x_0}=-\Delta|_{D_0}$. The operators associated with
equation \eqref{specproblem} form the family of self-adjoint
extensions of the positive symmetric operator $-\Delta_{x_0}$ (cf.
\cite{CdV}, section 1, p. 277).

We make the following conventions in the definition of the Green's
function.
\begin{definition}
Denote by $\sigma(-\Delta)$ the spectrum of $-\Delta$ on
$C^2(\T^2)$. We define the Green's function $G_\lambda(x;y)$ at
energy $\lambda\in\C\setminus\sigma(-\Delta)$ on $\T^2$ as the
integral kernel of the resolvent $(\Delta+\lambda)^{-1}$, that is
\begin{equation}
(\Delta+\lambda)^{-1}f(y)=\int_{\T^2}G_\lambda(x,y)f(x)dx.
\end{equation}
\end{definition}

In order to give a self-contained presentation of the theory of
self-adjoint extensions, we briefly recall the standard definitions
of the adjoint of an operator, symmetry and self-adjointness.
\begin{definition}
Let $H$ be a Hilbert space and $\Dom(B)\subset H$. Consider the
operator $B: \Dom(B)\to H$. We define
$$\Dom(B^*)=\lbrace y\in H \mid \exists a\in H:\;\forall x\in\Dom(B):\;\left\langle  Bx,y\right\rangle=\left\langle  x,a\right\rangle\rbrace$$ then we define the adjoint $B^*$ of $B$ as the map $B^*:\Dom(B^*)\to H$,
$$B^*y=a.$$ We call $B$ symmetric if $$\forall x,y\in\Dom(B): \left\langle  Bx,y\right\rangle=\left\langle  x,By\right\rangle.$$ We call $B$ self-adjoint if $B$  is symmetric and $$\Dom(B)=\Dom(B^*).$$
\end{definition}

We have the following well-known results from self-adjoint extension
theory which we summarize briefly. Proofs can be found in \cite{RS},
Chapter X.1. We will give the relevant references for each lemma.
\begin{definition}
Let $B$ be a densely defined symmetric operator on a Hilbert space
$H$. Denote its adjoint by $B^*$. Let $\eta\in\C\setminus\R$. The
deficiency spaces of $B$ at $\eta$ and $\bar{\eta}$ are defined as
\begin{equation}
D_\eta(B)=\ker\lbrace B^*-\eta\rbrace, \qquad
D_{\bar{\eta}}(B)=\ker\lbrace B^*-\bar{\eta}\rbrace.
\end{equation}
We refer to the members of a basis of a deficiency space as
deficiency elements.
\end{definition}

The following lemma is proven as part of Theorem X.1, p. 136 in
\cite{RS}.
\begin{lemma}
As a function of $\eta$, $\dim D_\eta(B)$ is constant  on the upper
(lower) complex half-plane.
\end{lemma}

We proceed with the definition of the deficiency indices of a closed
symmetric operator which indicate if the operator can be extended to
a self-adjoint operator.
\begin{definition}
If $\dim D_\eta(B)=m$ and $\dim D_{\bar{\eta}}(B)=n$ for nonnegative
integers $m$, $n$, we say that the operator $B$ has deficiency
indices $(m,n)$.
\end{definition}

The following lemma is Corollary (a), p. 141 in \cite{RS}.
\begin{lemma}
$B$ has deficiency indices $(0,0)$ if, and only if, $B$ is
self-adjoint.
\end{lemma}

If the deficiency indices are nonzero and equal, then a family of
self-adjoint extensions exists and can be constructed as follows.
This lemma combines Theorem X.2, p. 140 and Corollary (b), p. 141 in
\cite{RS}.
\begin{lemma}\label{sae}
If a closed positive symmetric operator $B$ has deficiency indices
$(n,n)$, $n\geq1$, then for each unitary map $U: D_\i(B)\to
D_{-\i}(B)$ there is a self-adjoint extension $B_U: D_U\to H$, where
\begin{equation}
D_U=\lbrace f=g+h+Uh \mid (g,h)\in \Dom(B) \times D_{\i}(B)\rbrace
\end{equation}
and $B_U$ acts as follows
\begin{equation}
B_U f=Bg+\i h-\i Uh.
\end{equation}
The operator $B_U$ has deficiency indices $(0,0)$. Conversely, every
self-adjoint extension of $B$ is of this form.
\end{lemma}

We apply Lemma \ref{sae} to construct a one parameter family of
self-adjoint extensions of the operator $-\Delta_{x_0}$. Denote the
domain of the closure of $-\Delta_{x_0}$ by $\tilde{D_0}$.
\begin{lemma}
The operator $-\Delta_{x_0}$ has deficiency indices $(1,1)$. The
corresponding deficiency elements are the Green's functions
$G_{\i}(x,x_0)$, $G_{-\i}(x,x_0)$.

The self-adjoint extensions of $-\Delta_{x_0}$ are given by the one
parameter family
\begin{equation}
-\Delta_{\varphi}:D_\varphi\to L^2(\T^2),\qquad \varphi\in(-\pi,\pi]
\end{equation}
where
\begin{equation}
D_\varphi=\lbrace g+cG_\i+ce^{\i\varphi}G_{-\i}:   g\in \tilde{D_0},
c\in \C\rbrace
\end{equation}
and
\begin{equation}
-\Delta_\varphi f=-\Delta g+c\i G_\i-c e^{\i\varphi} \i G_{-\i}.
\end{equation}
\end{lemma}
\begin{proof}
By definition of the Green's function we have that
\begin{equation}
\ker\lbrace(\Delta_{x_0})^*\pm\i\rbrace=\scrL\lbrace
G_{\pm\i}\rbrace.
\end{equation}
Hence $-\Delta_{x_0}$ has deficiency indices $(1,1)$ and we may
apply Lemma \ref{sae} to obtain the result.
\end{proof}
\begin{remark}
Note that $-\Delta_\pi$ recovers the Laplacian on $C^\infty(\T^2)$.
\end{remark}

Next we derive an equation for eigenvalues of the operator
$-\Delta_\varphi$, $\varphi\in(-\pi,\pi)$.
\begin{lemma}
Let $\varphi\in(-\pi,\pi)$. We have that $\lambda\notin
\sigma(-\Delta)$ is an eigenvalue of $-\Delta_{\varphi}$ if, and
only if,
\begin{equation}\label{speceq}
\sum_{\xi\in\lat}\left( \frac 1{|\xi|^2-\lambda} -
\frac{|\xi|^2}{|\xi|^4+1} \right)=c_0\tan(\varphi/2),
\end{equation}
where $$c_0=\sum_{\xi\in\lat}\frac{1}{|\xi|^4+1}.$$ The
corresponding eigenfunction is a multiple of $G_\lambda(x;x_0)$.
\end{lemma}
\begin{proof}
Let $f\in D_\varphi$ and $\left\|f\right\|_2=1$. Then $f$ must be of
the form
\begin{equation}
f=g+c G_\i+c e^{\i\varphi}G_{-\i}, \qquad g\in D_0, c\in \C.
\end{equation}

Let us first assume that $\lambda\notin\sigma(-\Delta)$ is an
eigenvalue of $-\Delta_\varphi$. We have
\begin{equation}
0=(\Delta_\varphi+\lambda)f =(\Delta+\lambda)g+c(\lambda-\i)G_\i+c
e^{\i\varphi}(\lambda+\i)G_{-\i}.
\end{equation}
We apply the resolvent $(\Delta+\lambda)^{-1}$ to both sides to
obtain
\begin{equation}\label{eq1}
0=g+c\frac{\lambda-\i}{\Delta+\lambda}G_\i+c
e^{\i\varphi}\frac{\lambda+\i}{\Delta+\lambda}G_{-\i}.
\end{equation}
In view of the iterated resolvent identity
\begin{equation}\label{itresolvent}
(\lambda\mp\i)\frac{1}{\Delta+\lambda}\frac{1}{\Delta\pm\i}=\frac{1}{\Delta\pm\i}-\frac{1}{\Delta+\lambda}
\end{equation}
we can rewrite equation \eqref{eq1} as
\begin{equation}\label{A.15}
0=g(x)+c(G_\i-G_\lambda)(x,x_0)+c
e^{\i\varphi}(G_{-\i}-G_\lambda)(x,x_0).
\end{equation}
In particular
\begin{equation}
f=g+c(G_\i +e^{\i\phi}G_{-\i}) = c(1+e^{\i\phi}) G_\lambda
\end{equation}
and so $f$ is a multiple of $G_\lambda$.

If we now take the limit $x\to x_0$ on the r.h.s. of \eqref{A.15} we
obtain
\begin{equation}\label{eq2}
0=\lim_{x\to x_0}(G_\i-G_\lambda)(x,x_0)+e^{\i\varphi}\lim_{x\to
x_0}(G_{-\i}-G_\lambda)(x,x_0)
\end{equation}
and note that $\lambda\notin\sigma(-\Delta)$ implies $c\neq0$ so we
may drop the constant. A simple rearrangement of this equation
yields
\begin{equation}\label{eq3}
\tan(\varphi/2)\lim_{x\to x_0}\Im G_\i(x,x_0) =\lim_{x\to
x_0}(G_\lambda-\Re G_\i)(x,x_0).
\end{equation}
In order to obtain equation \eqref{speceq} we require the following
$L^2$-identity for the Green's function $G_\lambda$ on $\T^{2}$
\begin{equation}
G_{\lambda}(x,x_{0})=-\frac{1}{4\pi^2}\sum_{\xi\in\lat}\frac{\exp(\i\xi\cdot
(x-x_{0}))}{|\xi|^{2}-\lambda},\qquad x\neq x_{0}.
\end{equation}
We rewrite the r.h.s. of \eqref{eq3} as
\begin{equation}
\begin{split}
&-\frac{1}{4\pi^2}\sum_{\xi\in\lat}e(\xi\cdot (x-x_{0}))\left\{\frac{1}{|\xi|^{2}-\lambda}-\Re\frac{1}{|\xi|^{2}-\i}\right\}\\
=&-\frac{1}{4\pi^2}\sum_{\xi\in\lat}e(\xi\cdot (x-x_{0}))\left\{\frac{1}{|\xi|^{2}-\lambda}-\frac{|\xi|^2}{|\xi|^4+1}\right\}\\
&\stackrel{x\to
x_0}{\longrightarrow}-\frac{1}{4\pi^2}\sum_{\xi\in\lat}
\left\{\frac{1}{|\xi|^{2}-\lambda}-\frac{|\xi|^2}{|\xi|^4+1}\right\}
.
\end{split}
\end{equation}
Finally, note that $$\lim_{x\to x_0}\Im
G_\i(x,x_0)=-\frac{1}{4\pi^2}\sum_{\xi\in\lat}\frac{1}{|\xi|^4+1}.$$

To see the reverse implication assume that $\lambda$ solves equation
\eqref{eq2}, a rearrangement of equation \eqref{speceq}. The r.h.s.
of equation \eqref{eq2} has singularities at points which are in
$\sigma(-\Delta)$, hence $\lambda\notin\sigma(-\Delta)$. We define
$$f_\lambda(x)=(G_\lambda-G_\i)(x,x_0)+e^{\i\varphi}(G_{\lambda}-G_{-\i})(x,x_0)$$
and observe that
\begin{equation}
(1+e^{\i\varphi})G_{\lambda}=f_\lambda+G_{\i}+e^{\i\varphi}G_{-\i}\in
D_{\varphi}
\end{equation}
because equation \eqref{eq2} implies $f_\lambda(x_0)=0$. The
iterated resolvent identity \eqref{itresolvent} implies
\begin{equation}
(\Delta_\varphi+\lambda)f_\lambda=(\Delta+\lambda)f_\lambda=-(\lambda-\i)G_\i-(\lambda+\i)G_\i
\end{equation}
and by the definition of the operator $\Delta_\varphi$ we have
\begin{equation}
(1+e^{\i\varphi})(\Delta_\varphi+\lambda)G_\lambda=(\Delta+\lambda)f_\lambda+(\lambda-\i)G_\i+(\lambda+\i)G_{-\i}=0.
\end{equation}
This concludes the proof.
\end{proof}

\end{document}